\documentclass[reqno,12pt,a4paper]{amsart}

\numberwithin{equation}{section}
\allowdisplaybreaks 

\usepackage[normalem]{ulem}

\usepackage{amssymb,eucal,mathrsfs,setspace,bm}
\usepackage[noadjust]{cite}

\usepackage{array}
\newcolumntype{C}{>{$}c<{$}} 

\usepackage[largesc,theoremfont]{newtxtext}      
\usepackage[libertine,cmbraces]{newtxmath} 
\begingroup
  \makeatletter
  \@for\theoremstyle:=definition,remark,plain\do{%
    \expandafter\g@addto@macro\csname th@\theoremstyle\endcsname{%
      \addtolength\thm@preskip{.5\baselineskip plus .2\baselineskip minus .2\baselineskip}
      \addtolength\thm@postskip{.5\baselineskip plus .2\baselineskip minus .2\baselineskip}
    }%
  }
\endgroup

\usepackage{enumitem}
\setitemize{leftmargin=*}   
\setenumerate{leftmargin=*, 
  label=\textup{(\alph*)},  
	ref=\textup{\alph*}}      

\usepackage{mathtools} 
\usepackage{graphicx}

\usepackage[colorlinks=true,citecolor=red,linkcolor=blue]{hyperref} 

\usepackage[capitalise,noabbrev]{cleveref} 

\renewcommand{\cong}{\simeq}

\DeclarePairedDelimiterX{\comm}[2]{\lbrack}{\rbrack}{#1 , #2}  
\DeclarePairedDelimiterX{\acomm}[2]{\lbrace}{\rbrace}{#1 , #2} 
\DeclarePairedDelimiterX{\inner}[2]{\langle}{\rangle}{#1 , #2} 
\DeclarePairedDelimiterX{\super}[2]{\lparen}{\rparen}{#1 \delimsize\vert \mathopen{} #2} 

\newcommand{\ra}{\rightarrow}

\newcommand{\fld}[1]{\mathbb{#1}}    
\newcommand{\alg}[1]{\mathfrak{#1}}  

\newcommand{\ZZ}{\fld{Z}}
\newcommand{\QQ}{\fld{Q}}

\newcommand{\CC}{\fld{C}}

\newcommand{\ah}{\alg{h}}

\theoremstyle{plain}
\newtheorem{theorem}{Theorem}
\newtheorem{corollary}[theorem]{Corollary}
\newtheorem{lemma}[theorem]{Lemma}
\newtheorem{proposition}[theorem]{Proposition}
\newtheorem{remark}[theorem]{Remark}

\usepackage[all]{xy}

\newcommand{\lam}{\lambda}

\newcommand{\ch}{\mathrm{ch}}

\newcommand{\hh}{\widehat{\mathfrak{h}}}

\newcommand{\an}{\alg{n}}

\newcommand{\Znn}{\mathbb{Z}_{\geq0}}
\newcommand{\cspan}{\mathrm{span}_{\CC}}

\newcommand{\com}[2]{\mathrm{Com}(#1,#2)}
\newcommand{\cB}{\mathcal{B}}
\newcommand{\cC}{\mathcal{C}}

\makeatletter
\renewcommand\author@andify{%
  \nxandlist {\unskip ,\penalty-1 \space\ignorespaces}%
    {\unskip {} \@@and~}%
    {\unskip \penalty-2 \space \@@and~}%
}
\makeatother

\begin{document}

\title[]{On the commutant of the principal subalgebra in the $A_1$ lattice vertex algebra}

\author[K~Kawasetsu]{Kazuya Kawasetsu}
\address[Kazuya Kawasetsu]{
Priority Organization for Innovation and Excellence\\
Kumamoto University\\
Kumamoto, Japan, 860-8555.
}
\email{kawasetsu@kumamoto-u.ac.jp}

\begin{abstract}
The coset (commutant) construction is a fundamental tool
to construct vertex operator algebras from known vertex operator algebras.
The aim of this paper is to provide a fundamental example of the commutants of vertex algebras ouside vertex operator algebras.
Namely, the  commutant $C$ of the {\em principal
subalgebra} $W$ of the $A_1$ lattice vertex operator algebra $V_{A_1}$ is investigated.
An explicit minimal set of generators of $C$, which consists of infinitely many elements and strongly generates $C$,
is introduced.
It implies that the algebra $C$ is not finitely generated.
Furthermore, Zhu's  Poisson algebra of $C$ is shown to be isomorphic to an infinite-dimensional algebra $\CC[x_1,x_2,\ldots]/(x_ix_j\,|\,i,j=1,2,\ldots)$.
In particular, the associated variety of $C$ consists of a point.
Moreover, $W$ and $C$ are verified to form a dual pair in $V_{A_1}$.

\end{abstract}

\maketitle

\onehalfspacing

\section{Introduction} \label{sec:intro}

The {\em coset}  construction is a fundamental tool to construct
 vertex operator algebras (VOAs) from known VOAs, since Goddard-Kent-Olive (GKO) construction of 2d rational conformal field theory (CFT).
 The cosets in 2d rational CFT are formulated as the {\em commutants} of VOAs in the theory of VOAs.
 Actually, the commutants are defined for more general {\em vertex algebras}, although
   there are no 
 detailed studies on the commutants of specific non-conformal vertex algebras to the best of our knowledge.
 The purpose of this paper is to establish a fundamental example
and make the situation clear to some extent.
 Namely, we study the commutant of the {\em principal subalgebra} of the $A_1$ lattice VOA, which is probably one of the most fundamental non-conformal vertex algebras.
Here, we call a vertex algebra $V$ non-conformal if there exist no 
Virasoro vectors $\omega$ in $V$ satisfying $L_{-1}=\partial$,
where $\partial$ is the derivation of $V$ and $L_{n}=\mathrm{Res}_z z^{n+1} Y(\omega,z)dz$. (In particular, non-conformal
vertex algebras are never VOAs.)

Let us recall the definition of commutants.
Let $V$ be a vertex algebra and $U$ a vertex subalgebra of $V$.
The {\em commutant} of $U$ inside $V$ is the 
vertex subalgebra
$
\com U V=\{v\in V\,|\, [u(n),v(m)]=0\ (u\in U, n,m\in\ZZ)\}
$
of $V$. 
In this paper, we investigate the commutant for the pair $W\subset V_{A_1}$, where the ambient algebra $V_{A_1}$ is the $A_1$ lattice VOA $V_{A_1}$ generated by two vertex operators $e^{\pm\alpha}(z)$,
and the subalgebra $W=\langle e^\alpha\rangle$ is the principal subalgebra of $V_{A_1}$ \cite{SF,MP} generated by the single vertex operator $e^\alpha(z)$. Let us set $C=\com W{V_{A_1}}$.

We first note some differences from usual situations.
The vertex subalgebra $W$ is neither conformal nor simple although $V_{A_1}$ is a simple VOA as usual.
Moreover, since $W$ is a commutative vertex algebra,
we have the inclusion $W\subset C$ and in particular,
the intersection $W\cap C$ is non-trivial.
In this sense, $C$ is no longer  a ``coset''.

Nevertheless, as usual, $W$ and $C$ are non-negatively graded by hamiltonian with finite-dimensional homogeneous pieces, 
since they are closed under the hamiltonian operator $L_0$  of $V_{A_1}$.
(In fact, they are closed under $L_n$ with $n\geq -1$, hence quasiconformal.)
In particular, we may define the $q$-characters of $W$ and $C$ as the graded traces of $L_0$. The characters, if normalized appropriately, are somehow modular functions of weight 0 on a congruence subgroup of $SL_2(\ZZ)$
(see \Cref{remchar}), like 
2d rational CFT.

As a main theorem, we give an explicit minimal set of generators of $C$, which strongly generates $C$ and consists of infinitely many generators.
We find the generators by applying the following simple realization of $C$.
Consider the dual lattice $A_1^\circ=\ZZ\varpi$ of $A_1$ with $\varpi=\alpha/2$.
The vertex operators $e^{\pm \varpi}(z)$ are formulated as a kind of parafermionic fields (those with non-necessarily integral powers of $z$) and generate the lattice {\em generalized vertex algebra} (GVA) $V_{A_1^\circ}$ \cite{DL}.
The {\em generalized principal subalgebra} $W^\circ=\langle e^\varpi\rangle$ is the sub\,GVA of $V_{A_1^\circ}$ generated by $e^\varpi(z)$ alone and
we have the realization $C=W^\circ\cap V_{A_1}$ by \cite{Kaw15}.
In particular, it follows that $C$ is generated by the  states 
$\Phi^\circ(n,m)=e^\varpi(-m-3/2)e^\varpi(-n-1)\bm1\in C$ with $m\geq n\geq0$ (see \Cref{prop:basisnew}).
The extremal vectors $\phi_n=\Phi^\circ(n,n)$ with $n\geq0$ generate $C$ minimally and strongly (\Cref{thm:genc}).

The minimality of the generators clearly implies that $C$ is not finitely generated  (\Cref{thm:notfin})
although $W$ and $V_{A_1}$ are {\em strongly} finitely generated. 
To the best of our knowledge, it is the first example of a non-finitely generated vertex subalgebra in $V_{A_1}$, although it may not be the easiest example: e.g., we can apply a similar method 
to show that $W\cap V_{2A_1}$ is not finitely generated.

The generators also allow us to determine {\em Zhu's Poisson algebra} $R_C$ of $C$ \cite{ZhuMod96}.
It relates vertex algebras to the Poisson geometry and other theories.
We show that $R_C$ is isomorphic to a quotient of the polynomial algebra of infinite variables:
$
R_C\cong \CC[x_1,x_2,\ldots]/(x_ix_j\,|\,i,j=1,2,\ldots),
$
with the trivial Poisson structure (\Cref{corrc}).
It follows that the {\em associated variety} $X_C=\mathrm{Specm}(R_C)$ is a singleton (\Cref{corxc}).
It may remind us that the {\em $C_2$-cofiniteness} condition \cite{ZhuMod96} is equivalent to  
finite strong generation and that the associated variety is a point
\cite{A12}.

Furthermore, $W$ and $C$ are verified to form a dual pair in $V_{A_1}$, that is, 
the commutant of $C$ coincides with $W$ 
 (\Cref{mainthm2}).

Recall that the $C_2$-cofiniteness of VOAs implies a modular invariance property of characters \cite{ZhuMod96,Miy} and  has been preserved under taking commutants. 
This preservation does not hold in our non-conformal setting:  $W$ and $V_{A_1}$ are $C_2$-cofinite but $C$ is not as $R_C$ is infinite-dimensional.
Note  that $C$ is not even quasi-lisse in the sense of \cite{AK}, a stage for a related modular invariance, since $C$ is not finitely generated.

We note that the principal subalgebras, or more generally, principal subspaces, are of independent interest and have been studied extensively. See e.g., \cite{CLM,CalLM,Geog,AKS,CoLM,FFJMM,BKP}, to name but a few.

Despite that $C$ is not a ``coset'' of $W$ in the sense that $C$ 
 intersects $W$ non-trivially, we still oddly have a demposition of the character of $V_{A_1}$ with respect to those of $W$-modules and restricted duals of $C$-modules,
 which 
 readily follows from \cite{BFL}, a combinatorial description of $V_{A_1}$ 
 related to their {\em Urod VOAs}, and the character formula
 of $C$-modules.
We hope to discuss this aspect in a forthcoming paper.

 We also note  that quotients of the restricted dual space $C^\vee$ 
 coincide with the level one $\widehat{sl}_2$ spaces of coinvariants 
 $L_{1,0}^{(N)}(\an)$ in \cite{FKLMM}. 
We hope to come back to this point in a forthcoming paper as well.

This paper is organized as follows.
In \Cref{subsecpre}, we recall the lattice VOA and GVA $V_{A_1}$ and  $V_{A_1^\circ}$, as well as the principal subalgebra $W$ and 
generalized one $W^\circ$.
In \Cref{subsecbi}-\Cref{subsecformula}, we introduce
bigradings and recall formulas for $V_{A_1^\circ}$, consequently for our subalgebras of $V_{A_1^\circ}$.
The commutant $C$ of $W$ is introduced in \Cref{subseccommutant}.
In \Cref{secgen}, we introduce the minimal set of (strong) generators of $C$
assuming some key propositions shown in \Cref{seckey}.
 In \Cref{secdual}, we verify that $(W,C)$ is a dual pair.
We determine $R_C$ and $X_C$ in \Cref{secrems} with concluding remarks.

\subsection*{Notations.}
All vector spaces in this paper are over the field of complex numbers $\CC$.
The set of integers is denoted by $\ZZ$.

\subsection*{Acknowledgements}

We thank Prof.~Atsushi Matsuo for helpful advice and careful reading of manuscripts. We also thank Profs.~Toshiyuki Abe, Masahiko Miyamoto and Hiroki Shimakura for fruitful advice and discussions.
This research is partially supported by
MEXT Japan ``Leading Initiative for Excellent Young Researchers (LEADER)'',
JSPS Kakenhi Grant numbers 19KK0065, 21K13775 and 21H04993.

\section{Preliminaries on the commutant $C$}\label{secpre}

\subsection{Lattice generalized vertex algebras and generalized principal subalgebras} \label{subsecpre}
We use the same notations as in \cite{Kaw15}.
In the following, we recall some notations necessary to this paper.

Let $A_1=\ZZ \alpha$ be the $A_1$ lattice with $(\alpha,\alpha)=2$
and $A_1^\circ=\ZZ\varpi$ the dual lattice of $A_1$ with $\varpi=\alpha/2$.
In particular, $(\varpi,\varpi)=1/2$.
We consider the abelian Lie algebra $\ah=\CC\otimes_{\ZZ} A_1$
with the induced bilinear form and the Heisenberg Lie algebra
$\hh=\ah[t,t^{-1}]\oplus \CC K$ with the central element $K$ and 
the Lie bracket $[h_1(n_1),h_2(n_2)]=n_1(h_1,h_2)\delta_{n_1+n_2,0}K$.
Here, we denote $ht^n=h(n)$ for any $h\in\ah$ and $n\in\ZZ$.
Let $\lam$ be an element of $A_1^\circ$. 
The bosonic Fock space is the induced module
$M(1,\lam)=U(\hh)\otimes_{U(\ah[t]\oplus \CC K)}\CC e^\lam$  generated by the highest weight vector $1\otimes e^\lam$,
where $\CC e^\lam$ is the one-dimensional module defined by $Ke^\lam=e^\lam$, $h(0)e^\lam=(h,\lam)e^\lam$
and $h(n)e^\lam=0$ for all $n\geq 1$.
We often simply write $e^\lam=1\otimes e^\lam\in M(1,\lam)$.


Let  $\beta$ be an element of $A_1^\circ$. We consider the {\em vertex operator}
$$
e^\beta(z)=\exp\left(\sum_{n\in\ZZ_{<0}}\frac{\beta(n)}{-n}z^{-n}\right)
\exp\left(\sum_{n\in\ZZ_{>0}}\frac{\beta(n)}{-n}z^{-n}\right)
 \otimes z^{\beta(0)}e_{\beta},
$$
acting on the Fock spaces.
It sends the space $M(1,\lam)$ into $M(1,\lam+\beta)[[z^{1/2},z^{-1/2}]]$ 
for any $\lam\in A_1^\circ$.
Here, $z^{\beta(0)} e^\lam=z^{(\beta,\lam)}e^\lam$ and $e_\beta e^\lam=e^{\beta+\lam}$.
We write the coefficients of $e^\beta(z)$ as $e^\beta(z)=\sum_{n\in\frac12\ZZ}e^\beta(n)z^{-n-1}$.

  
The fields $e^{\pm \alpha}(z)$ generate the lattice VOA $V_{A_1}$. 
The underlying vector space is
$$
V_{A_1}=\bigoplus_{n\in\ZZ}M(1,n\alpha),
$$
with the vacuum vector $\bm 1=e^0$ and the state-field correspondence $Y(\cdot,z): V_{A_1}\ra \mathrm{End}(V_{A_1})[[z,z^{-1}]]$ for $V_{A_1}$ is determined by 
$$
Y(\varpi(-1)\bm1,z)=\sum_{n\in\ZZ}\varpi(n)z^{-n-1},\qquad
Y(e^\beta,z)=e^\beta(z)\qquad (\beta\in A_1).
$$

The field $e^\alpha(z)$ alone generates the lattice principal subalgebra
$$
W=\langle e^\alpha\rangle\subset V_{A_1}.
$$
We have the following combinatorial basis of $W$:
the set 
$$
\cB=\{\Phi(n_1,\ldots,n_r)\,|\,n_r\geq \cdots \geq n_1\geq 0\}\subset W
$$
consisting of monomials
\begin{align*}
&\Phi(n_1,\ldots,n_r)=e^\alpha(-n_r-2r-1)e^\alpha(-n_{r-1}-2r+1)\cdots 
e^\alpha(-n_2-3)e^\alpha(-n_1-1)\bm1
\end{align*}
is a $\CC$-basis of $W$ \cite{SF}.

Historically, the space $W$ was first introduced in \cite{SF} as 
a subspace of the level one basic vacuum module  over the affine Kac-Moody algebra $A_1^{(1)}$ 
among other subspaces of standard modules ({\em principal subspaces}) and brought to the theory of vertex algebras in \cite{CLM}.
The lattice analogue of principal subspaces is introduced in \cite{MP} and the equivalent notion of free vertex algebras is studied in \cite{R}.

Similarly, the fields $e^{\pm \varpi}(z)$ generate the lattice GVA $V_{A_1^\circ}$. The differences are that the underlying vector space is
$$
V_{A_1^\circ}=\bigoplus_{n\in\ZZ}M(1,n\varpi),
$$
which includes $V_{A_1}$ as a vertex subalgebra, 
and that the state-field correspondence $Y(\cdot,z)$ for $V_{A_1^\circ}$
takes values in $\mathrm{End}(V_{A_1^\circ})[[z^{1/2},z^{-1/2}]]$, so that 
$V_{A_1^\circ}$ is not a vertex algebra but a GVA.

The field $e^\varpi(z)$ alone generates the lattice generalized principal subalgebra
$$
W^\circ=\langle e^\varpi\rangle\subset V_{A_1^\circ}.
$$
The set \cite{Kaw15}
$$
\cB^\circ=\{\Phi^\circ(n_1,\ldots,n_r)\,|\,n_r\geq \cdots \geq n_1\geq 0\}\subset W^\circ
$$
is a $\CC$-basis of $W^\circ$, where
\begin{align*}
&\Phi^\circ(n_1,\ldots,n_r)=e^\varpi(-n_r-(r+1)/2)e^\varpi(-n_{r-1}-r/2)\cdots 
e^\varpi(-n_2-3/2)e^\varpi(-n_1-1)\bm1.
\end{align*}

\subsection{Bigradings}\label{subsecbi}
We have the hamiltonian linear operator $H:V_{A_1^\circ}\ra V_{A_1^\circ}$ defined by  
$H (e^{\beta})=(\beta,\beta)/2$ for any $\beta\in A_1^\circ$
and the commutation relation $[H,\varpi(n)]=-n\varpi(n)$ for
all  $n\in\ZZ$.
 A non-zero vector $v$ is said to be a homogeneous vector of  {\em conformal weight} $\Delta_v\in\CC$ if $Hv=\Delta_v v$.
The conformal weight of the monomial
$$
v=\varpi(-n_1-1)\cdots \varpi(-n_s-1)e^{\beta} \quad 
(\beta\in  A_1^\circ, s\geq 0,  n_1,\ldots,n_s\in\ZZ).
$$
is $(\beta,\beta)/2+n_1+\cdots+n_s$.
Note that for any homogeneous $v\in V$ of conformal weight $\Delta_v$ and $n\in\CC$, we have the commutation relation
$$
[H,v(n)]=(\Delta_v-n-1)v(n).
$$
Here and after,  the $n$-th mode $v(n)$ is defined by
$Y(v,z)=\sum_{n\in\frac12\ZZ} v(n)z^{-n-1}$ for any element $v$.
For any vector subspace $M\subset V_{A_1^\circ}$,
we set $M_\Delta=\{v\in M\,|\, Hv=\Delta v\}$, the homogeneous space of conformal weight $\Delta$.

The non-zero elements of $M(1,\beta)$ is said to have {\em charge} $\beta$.
For any vector subspace $M\subset V_{A_1^\circ}$,
we write the homogeneous space of charge $\beta$ as
 $M^\beta=\{0\}\cup \{v\in M\,|\,v \mbox{ has charge }\beta\}$.
 In particular, $(V_{A_1^\circ})^\beta=M(1,\beta)$.

We say that a vector subspace $M\subset V_{A_1^\circ}$
is bigraded if $M=\bigoplus_{\beta\in A_1^\circ,\Delta\in \QQ} M^\beta_\Delta$ with $M^\beta_\Delta=M^\beta\cap M_\Delta$.

The vertex subalgebra $V_{A_1}$ of $V_{A_1^\circ}$ is bigraded with non-negative conformal weights:
$V_{A_1}=\bigoplus_{r\in\ZZ,n\in\Znn} (V_{A_1})^{r\alpha}_n$.
The vectors of conformal weight 0 are precisely the 
non-zero  scalar multiples of $\bm1$.

We have the inclusion $V_{A_1}\subset V_{A_1^\circ}$ and the direct sum decomposition $V_{A_1^\circ}=V_{A_1}\oplus V_{A_1+\varpi}$
as $V_{A_1}$-modules.
Here, we set $V_{A_1+\varpi}=\bigoplus_{r\in\ZZ}M(1,r\alpha+\varpi)$.
The conformal weights of $V_{A_1+\varpi}$ run over $\frac14+\Znn$.
The vectors of conformal weight $1/4$ are the linear sums of $e^{\pm\varpi}$ except for 0.
As a result, the whole space $V_{A_1^\circ}$ is bigraded:
$$
V_{A_1^\circ}=\bigoplus_{\substack{r\in\ZZ,\\ \Delta\in \Znn\cup(\frac14+\Znn)}} (V_{A_1^\circ})^{r\varpi}_\Delta,
$$

Since the elements of the bases $\cB$ and $\cB^\circ$ of $W$ and $W^\circ$  are bi-homogeneous, it follows that $W$ and $W^\circ$ are bigraded subalgebras of $V_{A_1^\circ}$:
$$
W=\bigoplus_{r,\Delta= 0}^\infty W^{r\alpha}_\Delta,
\qquad
W^\circ=\bigoplus_{\substack{r\in\Znn, \\ \Delta\in \Znn\cup(\frac14+\Znn)}} (W^\circ)^{r\varpi}_\Delta,
$$
Note that only non-negative  $r$ appear in the summation.

\subsection{Basic formulas}\label{subsecformula}
  
Let us write $V=V_{A_1^\circ}$.
In this Subsection, we collect basic necessary formulas for $V$.
For $\beta,\gamma\in A_1^\circ$, we set $\Delta(\beta,\gamma)=-(\beta,\gamma)+\ZZ\in \QQ/\ZZ$ and $\eta(\beta,\gamma)=e^{\pi\sqrt{-1}(\beta,\gamma)}\in \CC^\times$.
Notice that $\Delta(A_1,A_1^\circ)=\Delta(A_1^\circ,A_1)=\{\ZZ\}$ and $\Delta(A_1+\varpi,A_1+\varpi)=\{1/2+\ZZ\}$.

Let $\beta,\gamma,\delta$ be elements of $A_1^\circ$.
We have the {\em Borcherds identity}: for any $b\in V^\beta, c\in V^\gamma,d\in V^\delta$  and
$n\in \Delta(\beta,\gamma), k\in \Delta(\gamma,\delta), m\in \Delta(\beta,\delta)$,
\begin{align}\label{borcherds}
&\sum_{j\geq 0}\binom m j (b(n+j)c)(m+k-j)d\nonumber\\
&\qquad=\sum_{j\geq 0}(-1)^j\binom nj \Bigl(b(m+n-j)c(k+j)d-
\eta(\beta,\gamma)e^{\pi \sqrt{-1}n}c(n+k-j)b(m+j)d\Bigr).
\end{align}
The {\em derivation} (or translation operator) of $V$ is $\partial:V\ra V$ defined by $\partial a=a(-2)\bm1$ for all $a\in V$.
We have the {\em translation covariance} $[\partial,Y(a,z)]=\partial_zY(a,z)$ and $\partial\bm1=0$. 
In particular, we have $a(-n-1)\bm1=\frac1{n!}\partial^na$ 
and
\begin{equation}\label{derivnm}
(a(-n-1)\bm1)(m)=\binom{-m+n-1}{\ell}a(m-n),
\end{equation}
($n\in\ZZ_{\geq0}, m\in\CC$).

 Consider vectors $b\in V^\beta, c\in V^\gamma,d\in V^\delta$. 
 If $\Delta(\beta,\delta)=\ZZ$, that is, if at least one of $\beta,\delta$ belongs to $A_1$, we may substitute $m=0$ to 
 Borcherds identity \eqref{borcherds} and it results in
 the {\em associativity formula}
\begin{align}\label{associativity}
(b(n)c)(k)d
=\sum_{j\geq 0}(-1)^j\binom nj \Bigl(b(n-j)c(k+j)d-
\eta(\beta,\gamma)e^{\pi \sqrt{-1}n}c(n+k-j)b(j)d\Bigr).
\end{align}
for any
$n\in \Delta(\beta,\gamma)$ and $k\in \Delta(\gamma,\delta)$.
 
 Instead, if $\Delta(\beta,\gamma)=\ZZ$, 
 we may substitute $n=0$ to \eqref{borcherds} and obtain
 the {\em commutation relation}
\begin{align}\label{commutation}
(b(m)c(k)-\eta(\beta,\gamma)c(k)b(m))d=
\sum_{j\geq 0}\binom m j (b(j)c)(m+k-j)d
\end{align}
for any
$m\in \Delta(\beta,\delta)$ and $k\in \Delta(\gamma,\delta)$.
 
 All of the above are a kind of axioms of GVAs.
 We finally give some specific formulas for $V_{A_1^\circ}$:
\begin{align*}
&[\varpi(n),e^\beta(m)]=(\varpi,\beta)e^\beta(n+m)\quad (n\in\ZZ,m\in\frac12\ZZ),\\
&e^{\beta}(-(\beta,\gamma)-1)e^{\gamma}=e^{\beta+\gamma},\\
&e^{\beta}(-(\beta,\gamma)-1+n)e^{\gamma}=0\quad
( n\in\frac12\ZZ\setminus\ZZ_{<0}).
\end{align*}
 In particular, by using the last formula and \eqref{commutation}, we have the commutation relation 
$[e^\alpha(m),e^\alpha(n)]=0$ for all $m,n\in \ZZ$.
It implies that $W$ is a commutative vertex algebra, that is, 
$[u(m),v(n)]=0$ for all $u,v\in W$ and $m,n\in\ZZ$.
We also have the commutativity 
$e^\alpha(m)e^\varpi(k)=-e^\varpi(k)e^\alpha(m)$
 for all $m\in\ZZ$ and $k\in\frac12\ZZ$.

\subsection{The commutant $C$} \label{subseccommutant}

 Let us consider the commutant $C=\com W{V_{A_1}}$ of $W$ inside $V_{A_1}$.
 By \cite[Theorem 7.2]{Kaw15}, the invariant subspace
 $$
V_{A_1^\circ}^{W_+}:=\{v\in V_{A_1^\circ}\,|\, w(n)v=0\ (w\in W, n\geq 0)\}
$$
coincides with $W^\circ$.
It implies $C=W^\circ\cap V_{A_1}$ because of commutation relation \eqref{commutation}.
 Recall the basis elements $\Phi^\circ(n_1,\ldots,n_r)$ from \Cref{subsecpre}.
Since 
$$
\Phi^\circ(n_1,\ldots,n_r)\in 
\begin{cases}
V_{A_1}&\mbox{if $r$ is even}\\
V_{A_1+\varpi}& \mbox{if $r$ is odd},
\end{cases}
$$
we have the following $\CC$-basis $\cC$ of $C$:
\begin{equation}\label{basisc}
\cC=\{\Phi^\circ(n_1,\ldots,n_{2r})\,|\,r\in \ZZ_{\geq0}, n_{2r}\geq \cdots\geq n_1\geq 0\}.
\end{equation}
By using this basis, we see that $C$ is bigraded:
$$
C=\bigoplus_{r,\Delta= 0}^\infty C^{r\alpha}_\Delta,
$$
 and we obtain a fermionic character formula of $C$:
\begin{equation}\label{eqn:fullchar}
\ch(C)(z,q):=\sum_{r,\Delta=0}^\infty\dim(C^{r\alpha}_\Delta)q^\Delta z^r=\sum_{r=0}^\infty \frac{q^{r^2}z^{r}}{(q)_{2r}}.
\end{equation}
Here, $(q)_r=(q;q)_r=\prod_{n=1}^r(1-q^n)$ is the $q$-Pochhammer symbol.

\begin{remark}\label{remchar}
Let us consider the normalized ($q$-)character
\begin{equation}\label{eqn:qchar}
\widetilde \ch(C )(q):=q^{-1/40}\ch(C)(1,q)=q^{-1/40}\sum_{r=0}^\infty
\frac{q^{r^2}}{(q)_{2r}},
\end{equation}
where the second equality follows from \eqref{eqn:fullchar}.
By the fermionic sum formula \cite{KKMM}, 
it coincides with the character of the Virasoro minimal model 
$L(-3/5,-1/20)$. 
It is well known to be a modular form of weight 0 on $\Gamma(5)$.

The normalized character $\widetilde \ch(W)(q):=q^{-1/60}\ch(W)(1,q)$ is known to coincide with the character of $L(-22/5,-1/5)$ and is modular as well.
\end{remark}

\begin{remark}
The vertex algebras $V=W,C$ are not simple. We have
the maximal ideal $I=\bigoplus_{r= 1}^\infty V^{r\alpha}=\bigoplus_{\Delta= 1}^\infty V_\Delta$ of $V$. The 
quotient  is the trivial vertex algebra $\CC\bm 1$.
\end{remark}

\subsection{Quasiconformality}\label{sec:quasiconformality}
Recall that the Virasoro algebra 
$\mathrm{Vir}=\bigoplus_{n\in \ZZ} \CC L_n\oplus \CC C$
acts on $V_{A_1^\circ}$ via the stress-energy tensor
$$
T(z)=Y(\omega,z)=\sum_{n\in\ZZ}L_nz^{-n-2},
$$
with the Virasoro vector $\omega=\frac14 \alpha(-1)^2\bm 1$.
Note that $L_n\bm1=0$ for all $n\geq -1$.
Moreover, $L_{-1}$ and $L_0$ act as 
the derivation $\partial$ and 
  hamiltonian $H$ on $V_{A_1}^\circ$, respectively.
 The Virasoro vector $\omega$ also belongs to $V_{A_1}$ 
 and makes it into a VOA.

Since the vectors $e^\alpha$ and $e^\varpi$ are primary vectors
of $L_0$-weights 1 and $1/4$, respectively,
we have the commutation relations
$$
[L_n,e^\alpha(m)]=-me^\alpha(n+m),\quad
 [L_n,e^\varpi(m)]=\Biggl(-m-\frac34 n-\frac34\Biggr)e^\varpi(n+m).
$$
It follows that $W$ and $W^\circ$ are 
{\em quasiconformal}, which means that the positive part $\mathrm{Vir}_+=\bigoplus_{n\geq -1}\CC L_n$ of the Virasoro algebra acts on them and $L_{-1}$ coincides with $\partial$.
Since the Virasoro action preserves charges, it follows that $C$ is also quasiconformal.
Notice that being subspaces of $V_{A_1}$, both $W$ and $C$ are 
non-negatively graded by the hamiltonian $H=L_0$.

\begin{remark}
Although $V_{A_1}$ is a VOA with the Virasoro element $\omega$, the subalgebras $W$ and $C$ are non-conformal as follows.
Set  $V=W$ or $V=W^\circ$.
Suppose on the contrary that we have a Virasoro vector $\sigma\in V$ satisfying  $L_{-1}^\sigma=\partial$, where we denote
$Y(\sigma,z)=\sum_{n\in\ZZ}L_n^\sigma z^{-n-2}$. 
Since $e^\alpha\in V$ is a central element of $V$, 
we have $\partial e^\alpha=L^\sigma_{-1} e^\alpha=\sigma(0)e^\alpha=0$. 
It contradicts to $\partial e^\alpha=\alpha(-1)e^\alpha\neq 0$.
Thus, $W$ and $C$ are non-conformal.
\end{remark}

\section{A set of generators of $C$}\label{secgen}

 A subset $S\subset C$ {\em generates} $C$ if the monomials of the form
 $a_1(n_1)\cdots a_r(n_r)\bm1$ with $r\geq0$, $a_1,\ldots,a_r\in S$ and $n_1,\ldots,n_r\in \ZZ$ span $C$.
 Moreover, $S$ is said to {\em strongly} generate $C$ if the monomials with 
 $n_1,\ldots,n_r\leq -1$ already span $C$.
In this Section, we give a minimal set of (strong) generators of $C$.

Set $\phi_n:=\Phi^\circ(n,n)=e^\varpi(-n-3/2)e^\varpi(-n-1)\bm 1$  for $n\geq0$ and 
$$
S=\{\phi_n\,|\,n=0,1,2,\ldots\}\subset C^\alpha.
$$
The conformal weight of $\phi_n$ is $2n+1$ and charge is $\alpha$.
By using the standard argument, we have
$$
\phi_n=\sum_{k=0}^n\binom{1/2}k S_{n+k}(\alpha/2)S_{n-k}(\alpha/2)e^\alpha,
$$
where $S_k(x)=S_k(x(-1),x(-2),\ldots)$ is the elementary Schur polynomial defined by the formula
$$
\mathrm{exp}\left(\sum_{k<0}\frac{x(k)}{-k}z^{-k}\right)
=\sum_{k\geq 0}S_k(x)z^k.
$$
For example,
\begin{align*}
&\phi_0=e^\alpha, \quad \phi_1=\left(\frac3 {16}\alpha(-1)^2-\frac18\alpha(-2)\right)e^\alpha,\\
&\phi_2=\Bigl(-\frac1{64}\alpha(-4)-\frac5{96}\alpha(-3)\alpha(-1)
+\frac{15}{256}\alpha(-2)^2
+\frac7{256}\alpha(-2)\alpha(-1)^2\\
&\qquad+\frac{31}{3072}\alpha(-1)^4\Bigr)e^\alpha.
\end{align*}
Since $\phi_0=e^\alpha\in W$, we have $\phi_0(z)\phi_n(w)\sim 0$ for any $n\geq 0$.
However, $C$ is not a commutative vertex algebra since
$$
\phi_1(z)\phi_1(w)\sim-\frac{25}{32}\left(\frac{e^{2\alpha}(w)}{(z-w)^2}
+\frac{(\alpha(-1)e^{2\alpha})(w)}{z-w}\right).
$$

The following is the main result of  this Section.

\begin{theorem}\label{thm:genc}
The set $S$ strongly generates $C$. Moreover,
$S$ is a minimal set of generators in the sense that no proper subsets of $S$ generate $C$. 
\end{theorem}

 Note that the second statement is stated for 
 generators but not for strong generators.
In particular, we have the following corollary.

\begin{corollary}\label{thm:notfin}
The vertex algebra $C$ is not finitely generated.
\end{corollary}

\begin{proof}
Suppose on the contrary that $C$ is generated by a finite number of 
elements $v_1,\ldots,v_r$. Since we can express each $v_i$ with a finite number of elements $\phi_0,\ldots,\phi_{s_i}$ of $S$, by setting $s=\max\{s_1,\ldots,s_r\}$, we
have $V=\langle v_1,\ldots,v_r\rangle=\langle \phi_0,\ldots,\phi_s\rangle$, which contradicts to the minimality of $S$.
\end{proof}

The following propositions are the key propositions to prove \Cref{thm:genc}:

\begin{proposition}\label{prop:basisnew}
For any $r\geq 0$, the following vectors form a basis $\cC^{r\alpha}_{new}$ of $C^{r\alpha}$:
\begin{align}\label{eqn:basisnew}
&\Phi^\circ(n_{2r-1},n_{2r}+r-1)(-r)\Phi^\circ(n_{2r-3},n_{2r-2}+r)(-r+1)\cdots\\
&\quad\cdots \Phi^\circ(n_{2i-1},n_{2i}+i-1)(-i)\cdots
 \Phi^\circ(n_3,n_4+1)(-2)\Phi^\circ(n_1,n_2)(-1)\bm 1,\nonumber
\end{align}
with $n_{2r}\geq \cdots \geq n_1\geq 0$.
In particular, the collection $\cC_{new}:=\bigsqcup_{r=0}^\infty \cC_{new}^{r\alpha}$ is a basis of $C$.
 \end{proposition}
 
 The proposition in particular implies that the elements 
$\Phi^\circ (n,m)$ with $m\geq n\geq 0$ generate $C$.

\begin{proposition}\label{prop:basisc1}
The following set is a basis of the vector space $C^\alpha$:
$$
\langle S\rangle_\partial=
\{\partial^n\phi_m\,|\,n,m\geq0\}.
$$
\end{proposition}

 We postpone the proofs of these propositions to the next Section
 and first show \Cref{thm:genc} by using the propositions.

\begin{proof}[Proof of \Cref{thm:genc}]
Consider the basis element \eqref{eqn:basisnew} and
let $i=1,2,\ldots,r$ be an integer.
By \Cref{prop:basisc1}, the vector $\Phi^\circ(n_{2i-1},n_{2i}+i-1)$ is a linear sum of monomials of the form $\partial^\ell\phi_m$ with  $\ell,m\geq 0$.
Then the operator $\Phi^\circ(n_{2i-1},n_{2i}+i-1)(-i)$ is a linear sum of the operators of the form
$
\phi_m(-i-\ell)
$
with $\ell,m\geq 0$.
Thus, \eqref{eqn:basisnew} is a sum of monomials of the form $\phi_{m_1}(-\ell_1-1)\cdots \phi_{m_r}(-\ell_r-1)\bm1$ with $m_1\ldots,m_r,\ell_1,\ldots,\ell_r\geq 0$.
It implies that $C$ is strongly generated by $S$.

Let $S'$ be a proper subset of $S$ and $U=\langle S'\rangle_{VA}\subset C$.
Since the charges of elements of $S$ are all $\alpha$,
we see that $U^\alpha$ is spanned by $\langle S'\rangle_\partial$.
By \Cref{prop:basisc1}, we have $U^\alpha\subsetneq C^\alpha$.
Thus, $S'$ does not generate $C$, which is the minimality of $S$.
\end{proof}

Note that we clearly see by using  the combinatorial basis $\cB$ of $W$  that $W$ is strongly generated by $e^\alpha$.
The VOA $V_{A_1}$ is strongly generated by  $\{e^{\pm\alpha},\alpha(-1)\bm1\}$. 
They form a $sl_2$-triple with the 0-th mode as the Lie bracket.
Moreover, $V_{A_1}$ may be seen as the level one simple affine VOA
associated to this $sl_2$.

\section{Proofs of key propositions}\label{seckey}

In this Section, we show key propositions \Cref{prop:basisnew} and \Cref{prop:basisc1}.

We show some lemmas to prove \Cref{prop:basisnew}.
The following lemma follows immediately from associativity \eqref{associativity}.

\begin{lemma}\label{lem:borcx}
For any $n,\ell,k\in\ZZ$ and $d\in V_{A_1}$, we have
\begin{align*}
&(e^\varpi(-n-3/2)e^\varpi(-\ell-1)\bm1)(k)d\\
&\quad =\sum_{j\in\ZZ_{\geq0}} (-1)^j\binom{-n-3/2}j \binom{\ell-k-j-1}\ell
e^\varpi(-n-j-3/2)e^\varpi(-\ell+k+j)d\\
&\quad +\sum_{j\in\ZZ_{\geq0}} (-1)^{n+j}\binom{-n-3/2}j \binom{n+\ell-k+j+1/2}\ell
e^\varpi(-n-\ell+k-j-3/2)e^\varpi(j)d
\end{align*}
\end{lemma}

\begin{proof}
Substitute $b=e^\varpi$, $c=e^\varpi(-\ell-1)\bm1$ with $\beta=\gamma=\varpi$ and $\gamma\in L$ to  \eqref{associativity}
and replace $n$ with $-n-3/2$.
Then we have 
\begin{align*}
&(e^\varpi(-n-3/2)e^\varpi(-\ell-1)\bm1)(k)d\\
&\quad =\sum_{j\in\ZZ_{\geq0}} (-1)^j\binom{-n-3/2}j \{
e^\varpi(-n-j-3/2)(e^\varpi(-\ell-1)\bm1)(k+j)d\\
&\qquad -(-1)^{n+1}
(e^\varpi(-\ell-1)\bm1)(-n+k-j-3/2)e^\varpi(j)d\}.
\end{align*}
We then apply formula \eqref{derivnm}.
\end{proof}

We make use of  the following lexicographic order on the basis vectors $\Phi^\circ(n_1,\ldots,n_{2r})$ of $C^{r\alpha}$.
More precisely, we say $\Phi^\circ(n_1,\ldots,n_{2r})>\Phi^\circ(m_1,\ldots,m_{2r})$ if 
 $n_1>m_1$ or if there is $i\in\{1,\ldots,2r-1\}$ such that $n_j=m_j$ for all $1\leq j\leq i$ and $n_{j+1}>m_{j+1}$.

To see the ideas of the proof, we first show the vectors in \eqref{eqn:basisnew} with $r=2$ form a basis of $C^{2\alpha}$.

\begin{lemma}\label{lem:rec2}
The following vectors form a basis of $C^{2\alpha}$:
\begin{equation}\label{eqn:basisrec}
\Phi^\circ(n_3,n_4+1)(-2)\Phi^\circ(n_1,n_2),
\end{equation}
with $n_4\geq n_3\geq n_2\geq n_1\geq 0$.
\end{lemma}

\begin{proof}
Consider the vector in \eqref{eqn:basisrec} and write it by $\Psi(n_1,\ldots,n_4)$.
By applying \Cref{lem:borcx} to the operator
$$
\Phi^\circ(n_3,n_4+1)(-2)=(e^\varpi(-n_4-5/2)e^\varpi(-n_3-1)\bm1)(-2),
$$
we see that $\Phi^\circ(n_1,\ldots,n_4)$ is a sum of monomials of the form
\begin{align*}
&{\rm(i)}\ e^\varpi(-n_4-j-5/2)e^\varpi(-n_3-2+j)e^\varpi(-n_2-3/2)e^\varpi(-n_1-1)\bm1 \quad (j\geq0),\\
&{\rm(ii)}\ e^\varpi(-n_4-n_3-j-9/2)e^\varpi(j)e^\varpi(-n_2-3/2)e^\varpi(-n_1-1)\bm1 \quad (j\geq 0),
\end{align*}
with non-zero coefficients (but the vectors themselves can be zero if $j$ is large enough).
By writing it as the linear combination of the basis vectors of $C$, we see that
$$
\Psi(n_1,\ldots,n_4)=c\cdot \Phi^\circ(n_1,n_2,n_3,n_4)+(\mbox{lower terms})\quad (c\in\CC^\times).
$$
Thus, $\Psi(n_1,\ldots,n_4)$ with $n_4\geq n_3\geq n_2\geq n_1\geq 0$ form a basis of $C^{2\alpha}$.
\end{proof}

It is now easy to show \Cref{prop:basisnew}.

\begin{proof}[Proof of \Cref{prop:basisnew}]
As in the proof of \Cref{lem:rec2}, the element \eqref{eqn:basisnew}
has the form
$$
c\cdot \Phi^\circ(n_1,\ldots,n_{2r})+(\mbox{lower terms})\quad (c\in\CC^\times).
$$
Thus, we have proved the proposition.
\end{proof}

We next show a lemma to prove \Cref{prop:basisc1}. 

Recall the Virasoro modes $L_n$ from \Cref{sec:quasiconformality} and 
consider the $sl_2$-triple 
$$
e=-L_1, \quad h=-2L_0, \quad f=L_{-1},
$$
which satisfy $[e,f]=h, [h,e]=2e, [h,f]=-2f$.
Recall that $L_{-1}=\partial$ and $L_0=H$ on $V_{A_1^\circ}$.
The Lie algebra $sl_2=\langle e,h,f\rangle$ acts on the space $C^\alpha$, which enables us to apply the representation theory of $sl_2$ to study $C^\alpha$.
It turns out that $C^\alpha$ is an infinite direct sum of Verma modules.
Let $M(\lam)$ denote the Verma $sl_2$-module of highest weight $\lam\in\CC$.
Note that $C^\alpha_n$ is the $sl_2$-weight space of $C^\alpha$ of weight $-2n$.

\begin{lemma}\label{lemsl2}
As a $sl_2$-module, $C^\alpha$ decomposes into the sum of the form
$$
C^\alpha\cong \bigoplus_{n=0}^\infty M(-2(2n+1)).
$$
The highest weight vector $v_n$ of $M(-2(2n+1))$ has
conformal weight $2n+1$.
Moreover, for each $n\geq 0$, we have bases
$
\{\partial^{2n}v_0,\partial^{2n-2}v_1,\ldots,\partial^2v_{n-1},v_n\}
$
of $C^\alpha_{2n+1}$ and
$
\{\partial^{2n-1}v_0,\partial^{2n-3}v_1,\ldots,\partial v_{n-1}\}
$
of $C^\alpha_{2n}$.
\end{lemma}

\begin{proof}
We inductively prove the lemma.
First, the vector $v_0=\phi_0=e^\alpha\in C^\alpha_1$ is a highest weight vector 
of weight $-2$.
Since $C^\alpha_1$ is one-dimensional, the set $\{v_0\}$ is a basis of $C^\alpha_1$.
Note that the Verma $sl_2$-module $M(\lam)$ with a non-dominant 
weight $\lam$ is irreducible as well as projective.
Therefore, the submodule generated by $v_0$ is isomorphic to 
$M(-2)$ since the weight of $v_0$ is $-2$ and it is not dominant.
It then follows that the vectors $v_0,\partial v_0,\partial^2 v_0,\ldots$ are linearly independent.
Since $\dim C^\alpha_2=1$,
 the set $\{\partial v_0\}$ is a basis  $C^\alpha_2$.
Moreover, by the projectivity, $M(-2)=\langle v_0\rangle_{sl_2}$ is a direct summand of $C^\alpha$.
Now, since $\dim C^\alpha_3=2$, we have a highest weight vector
$v_1\in C^\alpha_3$ and $\{\partial^2v_0,v_1\}$ is a basis of $C^\alpha_3$.
Since the weight of $v_1$ is $-6$, which is not dominant,
we see that the submodule generated by $v_1$ is isomorphic to $M(-6)$ and it is a direct summand of $C^\alpha$.
Therefore, we have linearly independent set $\{v_0,\partial v_0,\ldots, v_1,\partial v_1,\ldots\}$.
Since $\dim C^\alpha_4=2$, the set $\{\partial^3v_0,\partial v_1\}$
is a basis of $C^\alpha_4$.
Then since $\dim C^\alpha_5=3$, we have a highest weight vector $v_2\in C^\alpha_5$ generating $M(-10)$.
By continuing the procedure, we have the lemma.
\end{proof}

\begin{proof}[Proof of \Cref{prop:basisc1}]
Let $n$ be a positive integer.
By \Cref{lemsl2}, the set $\{\partial^nv_m\,|\,n,m\geq 0\}$ is 
a basis of $C^\alpha$.
Therefore, it suffices to show that $v_n$ and $\phi_n$ are
congruent modulo $\partial (C^\alpha_{2n})$.
By \Cref{lemsl2}, we have $C^\alpha_{2n+1}=\partial (C^\alpha_{2n})\oplus \CC v_n$.
Since  $\partial:C^\alpha_{2n}\ra C^\alpha_{2n+1}$ is injective
by \Cref{lemsl2} and the set $\{\Phi^\circ(i,2n-1-i)\,|\,0\leq i\leq n-1\}$ is a basis of $C^\alpha_{2n}$ by \eqref{basisc}, 
the set $\{\partial\Phi^\circ(i,2n-1-i)\,|\,0\leq i\leq n-1\}$
is a basis of $\partial (C^\alpha_{2n})$.
Hence, it suffices to show that the system of vectors
$$
X=\{\partial\Phi^\circ(0,2n-1),\partial \Phi^\circ(1,2n-2),\ldots,
\partial\Phi^\circ(n-1,n),\phi_n\}
$$
is a basis of $C^\alpha_{2n+1}$.
We show this by expressing $X$ 
with respect to the basis 
$Y=\{\Phi^\circ(i,2n-i)\,|\,0\leq i\leq n\}$ of $C^\alpha_{2n+1}$.
Since 
\begin{equation*}
\partial \Phi^\circ(n,m)=\left(m+\frac32\right)\Phi^\circ(n,m+1)+(n+1)\Phi^\circ(n+1,m)\quad (n,m\geq0),
\end{equation*}
we have the following relation between $X$ and the basis $Y$:
\begin{align*}
&(\partial\Phi^\circ(0,2n-1),\partial \Phi^\circ(1,2n-2),\ldots,
\partial\Phi^\circ(n-1,n),\phi_n)\\
&\quad=
(\Phi^\circ(0,2n), \Phi^\circ(1,2n-1), 
\ldots, \Phi^\circ(n-1,n+1), \Phi^\circ(n,n))\cdot A
\end{align*}
with the lower triangular matrix
$$
A=
\begin{pmatrix}
2n+1/2& 0 &\cdots& \cdots &&&0\\
1&  2n-1/2&0&\cdots&&&0\\
0&  2&2n-3/2&\cdots&&&0\\
&&\cdots&&&\\
0&\cdots&&&&n+3/2&0\\
0&\cdots&&&&n&1\\
\end{pmatrix}.
$$
Since $\det A=(2n+1/2)(2n-1/2)\cdots(n+3/2)\cdot1\neq 0$,
the matrix $A$ is invertible.
Thus, $X$ is a basis of $C^\alpha_{2n+1}$, which completes the proof. 
\end{proof}

\section{The duality of $W$ and $C$}\label{secdual}

In this Section, we verify the following duality theorem.

\begin{theorem}\label{mainthm2}
The pair $(W,C)$ is a dual pair in $V_{A_1}$.
That is, $\com C {V_{A_1}}=W$ as well as $\com W{V_{A_1}}=C$.
\end{theorem}

Since we clearly have $\com C{V_{A_1}}\supset W$,
we just need to show the opposite inclusion.
We use the following: by \cite[Theorem 7.2]{Kaw15}, the invariant subspace of $W^\circ$
inside $V_{A_1}$ coincides with $W$:
\begin{equation}\label{invw}
W=\{v\in V_{A_1}\,|\,u(n)v=0\ (u\in W^\circ, n\geq 0)\}.
\end{equation}
(Note that $u(n)v=0$ if $n\not\in\ZZ$, $u\in W^\circ$ and $v\in W$. Moreover, $u(n)v$ does not necessarily belong to $V_{A_1}$.)
Then the following lemma is enough for the proof.

\begin{lemma}\label{annihilator}
For any $v\in V_{A_1}, r\in\ZZ$ and $p\in\ZZ_{\geq0}$,
the vector $e^\varpi(r-1/2)e^\varpi(p)v$ is realized
as a sum of vectors of the form $(e^\varpi(s-1/2)e^\varpi)(q)v$
with $s\in\ZZ$ and $q\in\ZZ_{\geq0}$.
\end{lemma}

We  first prove \Cref{mainthm2} by using the lemma.

\begin{proof}[Proof of \Cref{mainthm2}]
We show $\com C{V_{A_1}}\subset W$.
Let $v$ be an element of $\com C{V_{A_1}}$.
By \Cref{annihilator}, we have $e^\varpi(r-1/2)e^\varpi(p)v=0$
for any $r\in\ZZ$ and $p\in\Znn$.
It then follows that for each $p\in\Znn$, we have
$e^\varpi(p)v=0$ since otherwise, for sufficiently small $r\in\ZZ$,
we have $e^\varpi(r-1/2)(e^\varpi(p)v)\neq0$.
It implies that $v$ belongs to the invariant subspace of $W^\circ$ inside $V_{A_1}$ by associativity \eqref{associativity}.
Thus,  the assertion follows from \eqref{invw}.
\end{proof}

We now prepare to show \Cref{annihilator}.
We consider the operator
$$
\rho_{m,k;n}=\sum_{j= 0}^m\binom mj(e^\varpi(n-1/2+j)e^\varpi)(m+k-j)\in \mathrm{End}(V_{A_1}) \quad (m,k\in\Znn, n\in\ZZ).
$$
Note that $\rho_{m,k;n}v$ is a sum of vectors of the form
$(e^\varpi(s-1/2)e^\varpi)(q)v$ with $s\in\ZZ$ and $q\geq 0$,
 since $m+k-j\geq 0$ when $j$ runs over $0,1,\ldots,m$.

By Borcherds identity \eqref{borcherds}, we have
\begin{equation}\label{borchtemp}
\rho_{m,k;n}=\sum_{j\geq 0}(-1)^j\binom{n-1/2}j
\{e^\varpi(m+n-j-1/2)e^\varpi(k+j)-(-1)^n
e^\varpi(n+k-j-1/2)e^\varpi(m+j)\}.
\end{equation}

\begin{proof}[Proof of \Cref{annihilator}]
We may assume that  $v$ is a homogeneous vector of $V_{A_1}$
of conformal weight $\Delta_v\geq 1$.
In this case, we have
$e^\varpi(p)v=0$ if $p\geq \Delta_v$ 
and $e^\varpi(r-1/2)e^\varpi(p)v=0$ if $r+p\geq \Delta_v$,
since  $V_{A_1^\circ}$ is non-negatively graded.
Therefore, we may only consider the vectors of the form
$$
v_{N,p}=e^\varpi(N-p-1/2)e^\varpi(p)v\quad  (0\leq p< \Delta_v,
N<\Delta_v).
$$
We fix an integer $N<\Delta_v$.
It suffices to show that $v_{N,p}$ 
($0\leq p<\Delta_v$) are sums of vectors of the form $\rho_{m,k;n}v$ with $m,k\in\Znn, n\in\ZZ$.
We show it inductively from $p=\Delta_v-1$ to $p=0$.

(Step 1) We have
$v_{N,\Delta_v-1}=\rho_{\Delta_v,\Delta_v-1,N-2\Delta_v+1}v$
by using \eqref{borchtemp} and $e^\varpi(\Delta_v)v=0$.

(Step 2) We have
\begin{align*}
&\rho_{\Delta_v-1,\Delta_v-2,N-2\Delta_v+3}v
=e^\varpi(N-\Delta_v+3/2)e^\varpi(\Delta_v-2)v\\
&\quad -(-1)^{N-2\Delta_v+3}e^\varpi(N-\Delta_v+1/2)e^\varpi(\Delta_v-1)v\\
&\quad-(N-2\Delta_v+3-1/2)e^\varpi(N-\Delta_v+1/2)e^\varpi(\Delta_v-1)v\\
&\quad \in v_{N,\Delta_v-2}+\CC v_{N,\Delta_v-1}.
\end{align*}
By using (Step 1),   we see that $v_{N,\Delta_v-2}$ is a sum of 
$\rho_{\Delta_v-1,\Delta_v-2,N-2\Delta_v+3}v$
and $\rho_{\Delta_v,\Delta_v-1,N-2\Delta_v+1}v$.

In this way, in (Step $p$) $(p=3,\ldots,\Delta_v-1)$, we consider 
$\rho_{\Delta_v-p+1,\Delta_v-p,N-2\Delta_v+2p-1}v$ and show that $v_{N,\Delta_v-p}$
is a sum of $\rho_{\Delta_v-i+1,\Delta_v-i,N-2\Delta_v+2i-1}v$
for $i=1,2,\ldots,p-1$.
Thus, we have proved the assertion.
\end{proof}

\section{Zhu's Poisson algebra and the associated variety of $C$}
\label{secrems}

In this Section, we determine Zhu's Poisson algebras of $C$.
%

Let $V$ be a vertex algebra.
The Zhu's Poisson algebra of $V$ is the vector space $R_V=V/V(-2)V$ with 
the product $[u][v]=[u(-1)v]$ and the Poisson bracket
$\{[u],[v]\}=[u(0)v]$ ($u,v\in V$).
Here, $V(-2)V=\cspan\{u(-2)v\in V\,|\,u,v\in V\}$ and $[u]=u+V(-2)V\in R_V$.
A vertex algebra $V$ is called $C_2$-cofinite if 
$R_V$ is finite-dimensional \cite{ZhuMod96}.

The lattice VOAs are well known to be $C_2$-cofinite.
The principal subalgebra $W$ is $C_2$-cofinite
since $R_W\cong \CC[x]/(x^2)$.

We now determine Zhu's Poisson algebra $R_C$ of $C$.

\begin{theorem}
{\rm (i)} The  following vectors form a basis of $R_C$:
\begin{equation}\label{eqn:c2alg}
[\bm1],[\phi_0],[\phi_1],[\phi_2],\ldots.
\end{equation}
{\rm (ii)} The vector $[\bm1]$ is the identity element and  $[\phi_n]\cdot[\phi_m]=0$ for any $n,m\geq0$. The Poisson bracket $\{\cdot,\cdot\}$ is trivial.
\end{theorem}

\begin{proof}
By \Cref{prop:basisnew}, we see that the vector space $C^{(\geq 2)\alpha}:=\bigoplus_{r\geq 2}C^{r\alpha}$ is included in $C(-2)C$.
Since $C^\alpha(n)C^\alpha\subset C^{2\alpha}$ for any $n\in\ZZ$, 
it follows that $[\phi_n]\cdot[\phi_m]=0$ and $\{[\phi_n],[\phi_m]\}=0$ for any $n,m\geq 0$.
That $[\bm1]$ is the identity is well known.

We now show (i).
Since $C=\bigoplus_{r\geq0}C^{r\alpha}$ is nonnegatively graded by charges, we have
$$
C(-2)C=\partial C^\alpha\oplus C^{(\geq 2)\alpha}.
$$
By \Cref{prop:basisc1}, we see that the following vectors form a basis of $\partial C^\alpha$:
$$
\partial^n \phi_m\quad (n\geq 1,m\geq 0).
$$
Thus, \eqref{eqn:c2alg} form a basis of $R_C$.
\end{proof}

\begin{corollary}\label{corrc}
Zhu's Poisson algebra $R_C$ of $C$ is isomorphic to
$$
\CC[x_1,x_2,\ldots]/(x_ix_j\,|\,i,j=1,2,\ldots),
$$
with the trivial Poisson bracket.
\end{corollary}

\begin{corollary}\label{corxc}
Zhu's Poisson algebra $R_C$ of $C$ is local. 
Equivalently,
the associated variety $X_C=\mathrm{Specm}(R_C)$ consists of  a point. 
\end{corollary}

\begin{proof}
Since $I=\cspan\{[\phi_n]\,|\,n\geq 0\}$ is a maximal ideal,
it suffices to show that if $J$ is an ideal containing an element of the form
$v=1+u$ with $u\in I$,
then $1\in J$.
But we have $J\ni (1-u)v=1-u^2=1$, as desired.
\end{proof}

In the rest of the paper, we make some concluding remarks.
Recall that by \cite{CalLM}, the graded algebra
$\mathrm{gr}\,W$ of $W$ with respect to Li's filtration \cite{L05} on $W$ is
isomorphic to the jet lift of $R_W\cong \CC[x]/(x^2)$:
$$
\mathrm{gr}\,W\cong J_\infty(R_W),\quad J_\infty(R_W):=\CC[x,\partial x,\partial^2x,\ldots]/(x^2)_\partial,
$$
where $(x^2)_\partial$ is the differential ideal generated by $x^2$.
Note that there are always surjections from the jet sides to the graded sides \cite{L05}.

Regarding $C$, we remark that $\mathrm{gr}\,C$ is not isomorphic to $J_\infty(R_C)$. Here, we set
$$
J_\infty(R_C)=\CC[x_1,\partial x_1,\partial^2x_1,
\ldots,x_i,\partial x_i,\partial x_i^2,\ldots]/(x_ix_j\,|\,i,j=1,2,\ldots)_\partial.
$$
This is proved as follows.
By \cite{L} (see also \cite{AraKS}), we see that $J_\infty(R_C)$ is isomorphic to the graded algebra of the principal subalgebra 
$U=\langle e^{\beta_0},e^{\beta_1},\ldots\rangle$ of the lattice vertex algebra $V_L$ associated to the lattice 
$L=\bigoplus_{n=0}^\infty\ZZ \beta_n$
of infinite rank with the bilinear form defined by
$$
(\beta_i,\beta_j)=
\begin{cases}
2&(i=j)\\
1&(i\neq j)
\end{cases}.
$$
Therefore, we have the induced surjection $\mathrm{gr}\,U\twoheadrightarrow \mathrm{gr}\,C$.
Let us consider the gradation of $U$ with respect to the hamiltonian $H_U$ define by $H_U(e^{\beta_i})=(2i+1)e^{\beta_i}$,
to make the surjection a graded homomorphism.
Then  by using the combinatorial basis of $U$, we see that the character of $U$ with respect to $H_U$ has the form
$
\ch(U)(q)=\mathrm{Tr}_U\,q^{H_U}=1+q+q^2+2q^3+3q^4+5q^5+\cdots.
$ 
It does not coincide with the character of $C$: $\ch(C)(q)=1+q+q^2+2q^3+3q^4+4q^5+\cdots$.
Hence, $\mathrm{gr}\,C\not\cong J_\infty(R_C)$.

\begin{remark}\label{remeh}
It is shown in \cite{vEH} that the algebra $J_\infty (R_W)\cong \mathrm{gr}\,W$ is isomorphic to $\mathrm{gr}\,L(-22/5,0)$.
By considering the character correspondence (\Cref{remchar}), we 
may expect that $\mathrm{gr}\,C$ is  isomorphic to $\mathrm{gr}\,L(-3/5,0)$. However, it is not since 
 $\mathrm{gr}\,C$ is non-finitely generated whereas 
$\mathrm{gr}\,L(-3/5,0)$ is generated by one element.
\end{remark}

\begin{remark} The combinatorial basis $\cB^\circ$ of $W^\circ$
clearly implies that the quotient space 
$$ 
W^\circ/W^\circ(\leq -2)W^\circ
$$
of $W^\circ$ is finite-dimensional, 
where we set
$$
W^\circ(\leq -2)W^\circ=\cspan\{ u(n)v\,|\,n\in\frac12\ZZ, n\leq -2, u,v\in W^\circ\}.
$$
Therefore, it may be a natural question if $W^\circ$ gives an answer to the question in \Cref{remeh}, namely, if $W^\circ$  is related to the abelian intertwining algebra $L(-3/5,0)\oplus L(-3/5,3/4)$ under some filtration.
\end{remark}

\begin{remark}
Let $L=\ZZ\beta$ be a rank one lattice such that $2L$ is an integral lattice. Consider the subalgebra $U=\langle e^\beta\rangle \cap V_{2L}$.
Then $U$ has a similar set of minimal generators to that of $C$
and we can show this in a similar way as in this paper.
In particular, $W\cap V_{2A_1}$ is a non-finitely generated subalgebra of $V_{A_1}$.
There is another obvious way of generalization of the results of this paper that to consider the vertex algebra parts of generalized principal subspaces.
We hope to come back to these points in future works.
\end{remark}


\begin{thebibliography}{99999}

\bibitem[A]{A12}
Arakawa, T., ``A remark on the $C_2$-cofiniteness condition on vertex algebras.'' Math. Z. {\bf 270} (2012): 559--575.

\bibitem[AK]{AK}
Arakawa, T., and K. Kawasetsu, ``Quasi-lisse vertex algebras and modular linear differential equations.'' Lie Groups, Geometry, and Representation Theory: A Tribute to the Life and Work of Bertram Kostant (2018): 41--57.

\bibitem[AraKS]{AraKS}
Arakawa, T., K. Kawasetsu, and J. Sebag, ``A question of Joseph Ritt from the point of view of vertex algebras.'' J. Alg. {\bf 588} (2021): 118--128.


\bibitem[ArdKS]{AKS}
Ardonne, E., R.\,Kedem, and M.\,Stone, ``Fermionic characters of arbitrary highest-weight integrable $sl_{r+1}$-modules.'' Comm.\,Math.\,Phys.\ {\bf 264} (2006):
 427--464. 

%

\bibitem[BFL]{BFL}
Bershtein, M., B. Feigin, and A. Litvinov, ``Coupling of two conformal field theories and Nakajima-Yoshioka blow-up equations.'' Lett.  Math. Phys. {\bf 106} (2016): 29--56.

\bibitem[BKP]{BKP}
Butorac, M., S. Ko\u{z}i\'{c}, and M. Primc, ``Parafermionic bases of standard modules for affine Lie algebras.'' Math. Z. {\bf 298} (2021): 1003--1032.


\bibitem[CLM]{CLM}
Capparelli, S., J. Lepowsky, and A. Milas, 
``The Rogers-Ramanujan recursion and intertwining operators.'' Comm. Contemp. Math. {\bf 5} (2003): 947--966.



\bibitem[CalLM]{CalLM}
Calinescu, C., J. Lepowsky, and A. Milas, ``Vertex-algebraic structure of the principal subspaces of certain-modules, I: level one case.'' Int. J. Math. {\bf 19} (2008): 71--92.


\bibitem[CoLM]{CoLM}
Cook, W. J., H.\,Li, and K. C.\,Misra, ``A recurrence relation for characters of highest weight integrable modules for affine Lie algebras.'' Comm.\,Contem.\,Math.\ 
{\bf 9}  (2007): 121--133.


\bibitem[DL]{DL} Dong, C. and J.\,Lepowsky, ``Generalized vertex algebras and relative vertex operators.'' Springer, (1993).


\bibitem[FFJMM]{FFJMM}
Feigin, B., E.\,Feigin, M.\,Jimbo,  T.\,Miwa, and E.\,Mukhin, 
``Principal $\widehat{\mathfrak{sl}_3}$ subspaces and quantum Toda Hamiltonian.'' Algebraic analysis and around, Adv.\,Stud.\,Pure Math, {\bf 54} Math.\,Soc.\,Japan, Tokyo (2009): 109--166.

\bibitem[FKLMM]{FKLMM}
Feigin, B., R. Kedem, S. Loktev,  R. Miwa,  and E. Mukhin,
``Combinatorics of the $\widehat{sl}_2$ spaces of coinvariants.''
 Transf. groups {\bf 6}  (2001): 25--52.

\bibitem[G]{Geog}
Georgiev, G. N., ``Combinatorial constructions of modules for infinite-dimensional Lie algebras, I. Principal subspace.'' J.\,Pure Appl.\,Algebra {\bf 112}  (1996) 247--286.

\bibitem[K15]{Kaw15}
Kawasetsu, K., ``The free generalized vertex algebras and generalized principal subspaces.'' J. Alg. {\bf 444} (2015): 20--51.

\bibitem[KKMM]{KKMM}
Kedem, R., Klassen, T. R., McCoy, B. M., and E. Melzer, ``Fermionic sum representations for conformal field theory characters.'' Phys. Lett. B {\bf 307} (1993): 68--76.

\bibitem[Li05]{L05}
Li, Haisheng, ``Abelianizing vertex algebras.'' Comm. Math. Phys.  {\bf 259} (2005): 391--411.

\bibitem[L21]{L}
Li, Hao, ``Some remarks on associated varieties of vertex operator superalgebras.'' Euro. J. Math. {\bf 7} (2021): 1689--1728.

%

\bibitem[MP]{MP}
Milas, A., and M. Penn., ``Lattice vertex algebras and combinatorial bases: general case and W-algebras.'' N.Y. J. Math. {\bf 18} (2012): 621--650.

\bibitem[M]{Miy}
Miyamoto, M., ``Modular invariance of vertex operator algebras satisfying $C_2$-cofiniteness.'' Duke Math. J. {\bf 122} (2004): 51--91.

\bibitem[R]{R}
Roitman, M., ``Combinatorics of free vertex algebras.'' J. Alg.
{\bf 255} (2002): 297--323.

\bibitem[SF]{SF}
Stoyanovskii, A.\ V., and B. L. Feigin,  ``Functional models for representations of current algebras and semi-infinite Schubert cells.''
 Funct.\ Anal.\ Appl.\ {\bf 28} (1994): 55--72.
  
  \bibitem[vEH]{vEH}
  van Ekeren, J., and R. Heluani, ``Chiral homology of elliptic curves and the Zhu algebra.'' Comm. Math. Phys. {\bf 386} (2021): 495--550.
  
\bibitem[Z]{ZhuMod96}
Zhu, Y.,
\newblock ``Modular invariance of characters of vertex operator algebras.''
\newblock J. Amer. Math. Soc. {\bf 9} (1996): 237--302.


\end{thebibliography}
\end{document}